\documentclass{amsart}

\usepackage{setspace, amsmath, amsthm, amssymb, amsfonts, amscd, epic, graphicx, ulem, dsfont}
\usepackage[T1]{fontenc}
\usepackage{multirow}
\usepackage{bbm}

\newtheorem{thm}{Theorem}[section]
\newtheorem{lem}{Lemma}[section]
\newtheorem{pro}{Proposition}[section]

\newtheorem{cor}{Corollary}[section]

\newtheorem*{conj}{Conjecture}

\theoremstyle{remark}
\newtheorem*{rema}{Remark}

\theoremstyle{definition}
\newtheorem{defn}{Definition}[section]

\newcommand{\R}{\mathbb{R}}

\newcommand{\C}{\mathbb{C}}

\begin{document}

\title[commutativity up to a factor]{Commutativity up to a factor for bounded and unbounded operators}
\author[Chellali and Mortad]{Chérifa Chellali and Mohammed Hichem Mortad$^*$}

\thanks{$^*$ Corresponding author, who is partially supported by "Laboratoire d'Analyse Mathématique et Applications"}


\address{Department of
Mathematics, University of Oran, B.P. 1524, El Menouar, Oran 31000.
Algeria.\newline {\bf Mailing address (for the corresponding
author)}:
\newline Prof. Dr Mohammed Hichem Mortad \newline BP 7085
Es-Seddikia\newline Oran
\newline 31013 \newline Algeria}

\email{chchellali@gmail.com.}

\email{mhmortad@gmail.com.}

\begin{abstract}
In this paper, we further investigate the problem of commutativity
up to a factor (or $\lambda$-commutativity) in the setting of
bounded and unbounded linear operators in a complex Hilbert space.
The results are based on a new approach to the problem. We finish
the paper by a conjecture on the commutativity of self-adjoint
operators.
\end{abstract}

\maketitle

\section{Introduction}
Commutation relations between self-adjoint operators in a complex
Hilbert space are important in the interpretation of quantum
mechanical observables. They also play an important role when
analyzing their spectra. For more details see \cite{BBP} (and the
references therein), \cite{Davies-book-quantum-measurement} and
\cite{Putnam-book-1967}.

Recently, commutativity up to a factor has been given much attention
by many authors. See e.g. \cite{BBP}, \cite{CHO-JUN-TAKAEAKI},
\cite{khalagai-Kavila}, \cite{Mortad-ZAA-2010} and \cite{YANGDU}.

The purpose of the present paper is twofold. First, we recover known
results by examining bounded normal products of self-adjoint
operators. Second, we extend this method to unbounded operators.

Let us say a little more about details of this technique. It is
well-known that two bounded, normal and commuting operators have a
normal product. The proof uses the celebrated Fuglede theorem (we
note that this question has been generalized to the case of
unbounded operators in e.g. \cite{Mortad-Gustafson-Bull-SM},
\cite{Mortad-Demm-math}, \cite{Mortad-ex-Madani} and
\cite{Mortad-Operators-matrices}). In a very similar manner, we also
notice that -this time via the Fuglede-Putnam theorem- the product
of two anti-commuting normal operators remains normal. So, we
conjectured that the product of normal operators which commute up to
a factor would be normal. This is in effect the case and the reason
why we want to use the normality of the product in question is that
we may exploit results on the bounded normal product of self-adjoint
operators (as those in \cite{Mortad-PAMS2003} and
\cite{Mortad-IEOT-2009}, and the references therein). The advantage
of this approach is that it also extends to unbounded operators so
that we may again take advantage of the results in
\cite{Mortad-PAMS2003} and \cite{Mortad-IEOT-2009}.

To make the paper as self-contained as possible, we recall the
following results:

\begin{thm}\label{adjoints and closedness AB basic}
Let $A$ be a densely defined unbounded operator.
\begin{enumerate}
  \item $(BA)^*=A^*B^*$ if $B$ is bounded.
  \item $A^*B^*\subset (BA)^*$ for any densely unbounded $B$ and if $BA$ is
  densely defined.
  \item Both $AA^*$ and $A^*A$ are self-adjoint whenever $A$ is
  closed.
\end{enumerate}
\end{thm}

\begin{lem}[\cite{WEI}]\label{(AB)*=B*A*} If $A$ and $B$ are densely defined and $A$ is invertible with
inverse $A^{-1}$ in $B(H)$, then $(BA)^* =A^* B^*$.
\end{lem}

\begin{pro}[\cite{DevNussbaum-von-Neumann}]\label{Devinatz-Nussbaum-von Neumann: T=T1T2}
Let $A$, $B$ and $C$ be unbounded self-adjoint operators. Then
\[A\subseteq BC \Longrightarrow A=BC.\]
\end{pro}

\begin{thm}\label{Fuglede-Putnam}[Fuglede-Putnam theorem, for a proof see e.g. \cite{CON}]
If $A$ is a bounded operator and if $M$ and $N$ are normal
operators, then
\[AN\subseteq MA\Longrightarrow AN^*\subseteq M^*A\]
(if $N$ and $M$ are bounded, then we replace "$\subseteq$" by "=").
\end{thm}

\begin{thm}\label{AB normal implying AB self-adjoint}[\cite{Mortad-PAMS2003}]
Let $A$ and $B$ be two self-adjoint operators such that $AB$ is
normal. Then
\begin{enumerate}
  \item If $A$ and $B$ are both bounded, and if further $\sigma(A)\cap
\sigma(-A)\subseteq\{0\}$ or $\sigma(B)\cap
\sigma(-B)\subseteq\{0\}$, then $AB$ is self-adjoint.
  \item If only $B$ is bounded, and if $\sigma(B)\cap
\sigma(-B)\subseteq\{0\}$, then $AB$ is self-adjoint.
\end{enumerate}
\end{thm}

\begin{thm}\label{theorem BBP}[\cite{BBP}]
Let $A$, $B$ be bounded operators such that $AB\neq 0$ and
$AB=\lambda BA$, $\lambda\in\C^*$.  Then
\begin{enumerate}
  \item if $A$ or $B$ is self-adjoint, then $\lambda\in\R$;
  \item if both $A$ and $B$ are self-adjoint, then
  $\lambda\in\{-1,1\}$; and
  \item if $A$ and $B$ are self-adjoint and one of them is positive,
  then $\lambda=1$.
\end{enumerate}
\end{thm}

\begin{thm}\label{theorem Yang-Du}[\cite{YANGDU}]
Let $A$, $B$ be bounded operators such that $AB=\lambda BA\neq 0$,
$\lambda\in\C^*$.  Then
\begin{enumerate}
  \item if $A$ or $B$ is self-adjoint, then $\lambda\in\R$;
  \item if either $A$ or $B$ is self-adjoint and the other is normal, then
  $\lambda\in\{-1,1\}$; and
  \item if $A$ and $B$ are both normal, then $|\lambda|=1$.
\end{enumerate}
\end{thm}

\begin{thm}\label{mortad ZAA unbounded}[\cite{Mortad-ZAA-2010}]
Let $A$ be an unbounded operator and let $B$ be a bounded one.
Assume that $BA\subset \lambda AB\neq 0$ where $\lambda\in\C$. Then
\begin{enumerate}
  \item $\lambda$ is real if $A$ is self-adjoint.
  \item $\lambda=1$ if $0\not\in W(B)$ (the numerical range of $B$) and if $A$ is normal; hence
  $\lambda=1$ if $B$ is strictly positive and $A$ is normal.
  \item $\lambda\in\{-1,1\}$ if $A$ is normal and $B$ is self-adjoint.
\end{enumerate}
\end{thm}

In the end, we assume other notions and results on both bounded and
unbounded operators (some general textbooks are \cite{CON},
\cite{RUD} and \cite{WEI}). In particular, the reader should be
aware that invertible operators are taken to have an everywhere
defined bounded inverse, and that if $A$ and $B$ are two densely
defined unbounded operators, then
\[(\lambda A)B=A(\lambda B)=\lambda(AB)\]
whenever $\lambda\in \C^*$.

\section{Main Results}

\subsection{The Bounded Case}
\begin{defn}
Two bounded operators $A$ and $B$ are said to commute up to a factor
$\lambda$ (or $\lambda$-commute) if $AB=\lambda BA$ for some complex
$\lambda$.
\end{defn}

As alluded to in the introduction, we take on the problem of
commutativity up to a factor differently, that is, we first prove
that the product of two bounded normal $\lambda$-commuting operators
is normal. We have

\begin{thm}\label{commu lambda all bounded implies normality iff lambda anything}
Let $A$ and $B$ be two bounded normal operators such that
$AB=\lambda BA\neq 0$ where $\lambda\in\C$. Then $AB$ (and also
$BA$) is normal for any non-zero $\lambda$.
\end{thm}

\begin{proof}
Since $A$ and $B$ are both normal, so are $\lambda A$ and $\lambda
B$. So by Theorem \ref{Fuglede-Putnam}, we have
\[AB=\lambda BA\Longrightarrow AB^*=\overline{\lambda} B^*A \text{ and } A^*B=\overline{\lambda} BA^*.\]
Then we have on the one hand
\[(AB)^*AB=B^*A^*AB=B^*AA^*B=\overline{\lambda}B^*ABA^*=|\lambda|^2B^*BAA^*.\]
On the other hand we obtain
\[AB(AB)^*=ABB^*A^*=AB^*BA^*=\overline{\lambda}B^*ABA^*=|\lambda|^2B^*BAA^*.\]
Thus $AB$ is normal.
\end{proof}

\begin{rema}
Of course, thanks to Theorem \ref{theorem Yang-Du}, the condition
$|\lambda|=1$ was tacitly assumed in the previous theorem, but we
would have not needed it (cf. Theorem \ref{commu lambda unbounded
one implies normality iff lambda unit circle}).
\end{rema}

\begin{cor}
Let $A$ and $B$ be two bounded self-adjoint operators satisfying
$AB=\lambda BA\neq 0$ where $\lambda\in\C$. If either $\sigma(A)\cap
\sigma(-A)\subseteq\{0\}$ or $\sigma(B)\cap
\sigma(-B)\subseteq\{0\}$, then $\lambda=1$.
\end{cor}

\begin{proof}
By Theorem \ref{commu lambda all bounded implies normality iff
lambda anything}, $AB$ and $BA$ are normal. By Theorem \ref{AB
normal implying AB self-adjoint}, $AB$ and $BA$ are self-adjoint.
Hence
\[BA=(AB)^*=AB=\lambda BA,\]
yielding $\lambda=1$. The proof is thus complete.
\end{proof}

\begin{cor}\label{corollary all bounded one positive other s.a.}
Let $A$ and $B$ be two bounded self-adjoint operators satisfying
$AB=\lambda BA\neq 0$ where $\lambda\in\C$. Then $\lambda=1$ if one
of the following occurs:
\begin{enumerate}
  \item $A$ is positive;
  \item $-A$ is positive;
  \item $B$ is positive;
  \item $-B$ is positive.
\end{enumerate}
\end{cor}

\begin{cor}
If both $A$ and $B$ are self-adjoint (and bounded), then
\[AB=\lambda BA\Longrightarrow \lambda \in\{-1,1\}.\]
\end{cor}

\begin{proof}
We have
\[A^2B=\lambda ABA=\lambda^2BA^2.\]
Since $A$ is self-adjoint, $A^2$ is positive so that Corollary
\ref{corollary all bounded one positive other s.a.} gives
$\lambda^2=1$ or $\lambda \in\{-1,1\}$.
\end{proof}

\subsection{A digression}

\begin{defn}
Two bounded operators $A$ and $B$ are said to \textbf{commute
unitarily} if $AB=U BA$ for some unitary operator $U$.
\end{defn}

\begin{pro}
Let $A$ and $B$ be two bounded normal operators such that $AB=UBA$
for some unitary operator $U$. If $U$ commutes with $B$, then $UB$
and $AB$ are both normal.
\end{pro}

We omit the proof as it is very similar to that of Theorem
\ref{commu lambda all bounded implies normality iff lambda anything}
and hence we leave it to the interested reader.

\subsection{The Unbounded Case}

We may split the main result in this subsection into two parts for
the first one of the two is important in its own right. It also
generalizes known results for two normal operators (where at least
one of them is bounded, see e.g.  \cite{Mortad-Gustafson-Bull-SM}
and \cite{Mortad-Demm-math}).

\begin{thm}\label{commu lambda unbounded one implies normality iff lambda unit circle}
Let $A$ and $B$ be two normal operators where $B$ is bounded. Assume
that $BA\subset \lambda AB\neq 0$ where $\lambda\in \C$. Then $AB$
is normal iff $|\lambda|=1$.
\end{thm}

\begin{proof}
First, and since $A$ is closed and $B$ is bounded, $AB$ is
automatically closed.

Since $A$ is normal, so is $\lambda A$. Hence the Fuglede-Putnam
(see e.g. \cite{CON}) theorem gives
\[BA\subset \lambda AB\Longrightarrow BA^*\subset \overline{\lambda} A^*B \text{ or }\lambda B^*A\subset AB^*.\]
Using the above "inclusions" we have on the one hand
\begin{align*}
(AB)^*AB&\supset B^*A^*AB\\
&=B^*AA^*B ~\text{(since $A$ is normal)}\\
&\supset \frac{1}{\overline{\lambda}}B^*ABA^*\\
&\supset \frac{1}{\overline{\lambda}\lambda}B^*BAA^*\\
&=\frac{1}{|\lambda|^2}B^*BAA^*.
\end{align*}
Since $A$ and $AB$ are closed, and $B$ is bounded, all of
$(AB)^*AB$, $A^*A$ and $B^*B$ are self-adjoint so that "adjointing"
the previous inclusion yields
\[(AB)^*AB\subset \frac{1}{|\lambda|^2}AA^*B^*B.\]
As $|\lambda|$ is real, the conditions of Proposition
\ref{Devinatz-Nussbaum-von Neumann: T=T1T2} are met and we finally
obtain
\begin{equation}\label{equation Ab*Ab 1 B bounded}
(AB)^*AB=\frac{1}{|\lambda|^2}AA^*B^*B.
\end{equation}

On the other hand, we may write
\begin{align*}
AB(AB)^*&\supset ABB^*A^*\\
&=AB^*BA^* ~\text{(because $B$ is normal)}\\
&\supset \lambda B^*ABA^*\\
&=B^*(\lambda AB)A^*\\
&\supset B^*BAA^*.
\end{align*}
As above, we obtain
\[AA^*B^*B\supset AB(AB)^*\]
and by Proposition \ref{Devinatz-Nussbaum-von Neumann: T=T1T2} we
end up with
\begin{equation}\label{equation Ab*Ab 2 B bounded}
AB(AB)^*=AA^*B^*B.
\end{equation}
Accordingly, we clearly see that $AB$ is normal iff $|\lambda|=1$,
completing the proof.
\end{proof}

\begin{cor}
Let $A$ and $B$ be two normal operators where $B$ is bounded. If
$BA\subset \lambda AB\neq 0$ where $|\lambda|=1$, then
\[\overline{BA}=\lambda AB,\]
where $\overline{BA}$ denotes the closure of the operator $BA$.
\end{cor}

\begin{proof}
Since $BA\subset \lambda AB$ and $|\lambda|=1$, $AB$ is normal by
Theorem \ref{commu lambda unbounded one implies normality iff lambda
unit circle}.

Since $B$ is bounded and $A$ is densely defined, $BA$ too is densely
defined and so it has a unique adjoint. Hence
\[BA\subset \lambda AB\Longrightarrow \overline{\lambda} B^*A^*=B^*(\lambda A)^*\subset (\lambda AB)^*\subset (BA)^*=A^*B^*\]
or
\[B^*A^*\subset \frac{1}{\overline{\lambda}}A^*B^*.\]

Since $B^*$ is bounded and $|\frac{1}{\overline{\lambda}}|=1$,
Theorem \ref{commu lambda unbounded one implies normality iff lambda
unit circle} applies again and yields the normality of $A^*B^*$.
Whence $(BA)^*$ is normal as $(BA)^*=A^*B^*$ and so $(BA)^{**}$ is
normal. Now since $\lambda AB$ is normal, it is closed so that
$BA\subset \lambda AB$ tells us that $BA$ is closeable. Therefore,
$\overline{BA}=(BA)^{**}$. It follows that $\overline{BA}$ is
normal, and that
\[BA\subset \overline{BA}\subset \lambda AB.\]

Finally, as normal operators are maximally normal, we obtain that
\[\overline{BA}=\lambda AB,\]
establishing the result.
\end{proof}

\begin{pro}\label{commutatvity up to fcator unbounded implies lambda unit circle
proposition} Let $A$ and $B$ be two self-adjoint operators where $B$
is bounded. Assume that $BA\subset \lambda AB\neq 0$ where
$\lambda\in \C$. Then $AB$ is normal for any $\lambda$.
\end{pro}

\begin{proof}
Since $A$ and $B$ are self-adjoint, $BA\subset \lambda AB$ implies
the following three "inclusions"
\[BA\subset \overline{\lambda} AB,~\lambda BA\subset AB \text{ and } \overline{\lambda} BA\subset AB.\]
Proceeding as in the proof of Theorem \ref{commu lambda unbounded
one implies normality iff lambda unit circle}, we obtain
\[(AB)^*AB\supset |\lambda|^2B^2A^2 \text{ and } AB(AB)^*\supset |\lambda|^2B^2A^2.\]
Again as in the proof of Theorem \ref{commu lambda unbounded one
implies normality iff lambda unit circle}, $AB$ is normal.
\end{proof}

Now we apply the foregoing results to give spectral properties of
$\lambda$-commuting self-adjoint operators (in the unbounded case).

\begin{cor}\label{corollary unbounded B asymmetric spectrum}
Let $A$ and $B$ be self-adjoint operators where $B$ is bounded.
Assume that $BA\subset \lambda AB\neq 0$ where $\lambda\in \C$. If
further $\sigma(B)\cap \sigma (-B)\subseteq \{0\}$, then
$\lambda=1$.
\end{cor}

\begin{rema}
The previous result generalizes 2) of Theorem \ref{mortad ZAA
unbounded}. Besides the proof of Theorem \ref{mortad ZAA unbounded}
contained a small error which, by the present result, has now been
fixed.
\end{rema}

\begin{proof}
By Proposition \ref{commutatvity up to fcator unbounded implies
lambda unit circle proposition}, $AB$ is normal. Thanks to the
condition on the spectrum of $B$ and Theorem \ref{AB normal implying
AB self-adjoint} we get that $AB$ is self-adjoint. Hence $(AB)^*=AB$
so that
\[AB=(AB)^*\subset \frac{1}{\lambda}AB.\]
But $D(AB)=D(\alpha AB)$ for any $\alpha\neq 0$. Therefore,
\[AB=\frac{1}{\lambda}AB \text{ or merely } \lambda=1.\]
\end{proof}

\begin{cor}\label{corollary unbounded B positive}
Let $A$ and $B$ be self-adjoint operators where $B$ is bounded.
Assume that $BA\subset \lambda AB\neq 0$ where $\lambda\in \C$. Then
$\lambda=1$ if $B$ (or $-B$) is positive.
\end{cor}

\begin{cor}
Let $A$ and $B$ be self-adjoint operators where $B$ is bounded.
Assume that $BA\subset \lambda AB\neq 0$ where $\lambda\in \C$. Then
$\lambda \in\{-1,1\}$.
\end{cor}

\begin{proof}
We may write
\[B^2A\subset \lambda BAB\subset \lambda^2 AB^2.\]
Since $B$ is self-adjoint, $B^2$ is positive so that Corollary
\ref{corollary unbounded B positive} applies and gives $\lambda^2=1$
or $\lambda \in\{-1,1\}$.
\end{proof}

We finish this paper with the case of two unbounded operators. As
should have been expected, this case is quite delicate to handle
unless strong assumptions are made. But first, we start by a version
of the Fuglede-Putnam theorem.

\begin{thm}\label{Fuglede-Putnam all unbounded invertible}
Let $A$, $N$ and $M$ be three unbounded invertible operators on a
Hilbert space such that $N$ and $M$ are normal. If $AN=MA$, then
\[A^*M=NA^* \text{ and } AN^*= M^*A.\]
\end{thm}

\begin{proof}
We have

\[AN=MA \Longrightarrow A^{-1}M\subset NA^{-1}.\]
Since $A^{-1}$ is bounded, by Theorem \ref{Fuglede-Putnam}, we have
$A^{-1}M^*\subset N^*A^{-1}$  and hence
\[M^*A\subset AN^*.\]
By taking adjoints (and applying Lemma \ref{(AB)*=B*A*}) we obtain
\[NA^*\subset A^*M.\]

Now since $A^{-1}M\subset NA^{-1}$, we can get $N^{-1}A^{-1}\subset
A^{-1}M^{-1}$. But all operators (in the previous inclusion) are
bounded. Therefore, we get

\[N^{-1}A^{-1}= A^{-1}M^{-1}.\]

Applying the Fuglede-Putnam theorem (the all-bounded-operators
version) once more yields
\[(N^{-1})^*A^{-1}= A^{-1}(M^{-1})^*.\]

Whence
\[(M^{-1})^*A\subset A (N^{-1})^* \text{, and thus } AN^*\subset
M^*A.\] By Lemma \ref{(AB)*=B*A*} again, $A^*M\subset NA^*$. Thus
$AN^*= M^*A$ and $A^*M=NA^*$, establishing the result.
\end{proof}

\begin{cor}
If $A$ and $B$ are two unbounded normal and invertible operators
such that $AB=\lambda BA$, then
\[A^*B=\overline{\lambda}BA^* \text{ and } AB^*=\overline{\lambda} B^*A.\]
\end{cor}

With Lemma \ref{(AB)*=B*A*} and Theorem \ref{Fuglede-Putnam all
unbounded invertible} in hand, we may just mimic the proof of
Yang-Du (in \cite{YANGDU}) to prove the following result:

\begin{thm}
Let $A$, $B$ two unbounded invertible operators such that
$AB=\lambda BA\neq{0}$ , $\lambda \in\mathbb{C}$. Then
\begin{enumerate}
  \item If $A$ or $B$ is self-adjoint, then $\lambda\in\mathbb{R}$.
  \item If either $A$ or $B$ is self-adjoint and the other is normal, then   $\lambda\in \{-1,
  1\}$.
  \item If both $A$ and $B$ are normal, then $|\lambda|=1$.
\end{enumerate}
\end{thm}

\section{A Conjecture}
In Corollary \ref{corollary unbounded B asymmetric spectrum} (for
example), we said nothing about the spectrum of $A$. This is in fact
due to an (a natural) open question from \cite{Mortad-PAMS2003}
which the corresponding author of this paper has been working on
lately. Let us state it as a conjecture:
\begin{conj}
Let $A$ and $B$ be two self-adjoint operators such that only $B$ is
bounded and $A$ is positive. Then $AB$ is self-adjoint whenever $AB$
is normal.
\end{conj}

Neither a proof nor a counterexample have been reached yet. However,
we can state the following:
\begin{enumerate}
  \item Let $A=-f''$ be defined on $H^2(\R)$ (the Sobolev
  space, which is dense in $L^2(\R)$). Then $A$ is an unbounded self-adjoint, and positive operator in
  $L^2(\R)$. Let $B$ be a multiplication operator by an essentially bounded real-valued function
  $\varphi$ on $\R$. Hence $B$ is bounded and self-adjoint on
  $L^2(\R)$.

  This "counterexample" does not work for the conditions of the
  conjecture will force $\varphi$ to vanish (and so $AB=0)$.
  \item The conjecture is true with $BA$ in lieu of $AB$. But, the
  normality of $BA$ is stronger than that of $AB$ because $BA$ normal
  will then imply that $AB$ is normal too, and $AB=BA$! Details will
  appear in another paper.
  \item The conjecture is true if one assumes further that $BA$ is
  closed (in such case this will follow from the previous point).
  \item The conjecture seems to be a hard one. Indeed, the Fuglede
  (-Putnam) theorem is the tool par excellence when dealing with
  products involving normal (bounded or unbounded) operators.
  However, none of the
  known versions of the Fuglede theorem (such as \cite{FUG}, \cite{Mortad-PAMS2003} and \cite{Mortad-Fuglede-Putnam-CAOT-2012}) helps us in the
  proof to get $BA\subset AB$, a sufficient condition to make $AB$ self-adjoint.
\end{enumerate}

\section*{Acknowledgements}
The authors wish to thank both anonymous referees for their
meticulousness.

\bibliographystyle{amsplain}

\end{document}